\documentclass[a4paper,twoside]{article}
\usepackage{a4}
\usepackage{amssymb}
\usepackage{amsmath}
\usepackage{upref}
\usepackage{url}
\usepackage[active]{srcltx}
\usepackage[pagebackref,colorlinks,citecolor=blue,linkcolor=blue,urlcolor=blue]{hyperref}
\allowdisplaybreaks[2] 
\newcount\minutes \newcount\hours
\hours=\time
\divide\hours 60
\minutes=\hours
\multiply\minutes -60
\advance\minutes \time
\newcommand{\klockan}{\the\hours:{\ifnum\minutes<10 0\fi}\the\minutes}
\newcommand{\tid}{\today\ \klockan}
\newcommand{\prtid}{\smash{\raise 10mm \hbox{\LaTeX ed \tid}}}
\renewcommand{\prtid}{}
\makeatletter
\def\sectionmark#1{} 
\def\subsectionmark#1{}
\newcommand{\sectnr}{\ifnum \c@secnumdepth >\z@
                 \thesection.\hskip 1em\relax \fi}
\def\@evenhead{\footnotesize\rm\thepage\hfil\leftmark\hfil\llap{\prtid}}
\def\@oddhead{\footnotesize\rm\rlap{\prtid}\hfil\rightmark\hfil\thepage}
\def\tableofcontents{\section*{Contents} 
 \@starttoc{toc}}
\makeatother
\makeatletter
\def\@biblabel#1{#1.}
\makeatother
\makeatletter
\let\Thebibliography=\thebibliography
\renewcommand{\thebibliography}[1]{\def\@mkboth##1##2{}\Thebibliography{#1}
\addcontentsline{toc}{section}{References}
\frenchspacing 
\setlength{\@topsep}{0pt}
\setlength{\itemsep}{0pt}%
\setlength{\parskip}{0pt plus 2pt}%
}
\makeatother
\makeatletter
\def\mdots@{\mathinner.\nonscript\!.%
 \ifx\next,.\else\ifx\next;.\else\ifx\next..\else
 \nonscript\!\mathinner.\fi\fi\fi}
\let\ldots\mdots@
\let\cdots\mdots@
\let\dotso\mdots@
\let\dotsb\mdots@
\let\dotsm\mdots@
\let\dotsc\mdots@
\def\vdots{\vbox{\baselineskip2.8\p@ \lineskiplimit\z@
    \kern6\p@\hbox{.}\hbox{.}\hbox{.}\kern3\p@}}
\def\ddots{\mathinner{\mkern1mu\raise8.6\p@\vbox{\kern7\p@\hbox{.}}%
    \raise5.8\p@\hbox{.}\raise3\p@\hbox{.}\mkern1mu}}
\makeatother
\makeatletter
\let\Enumerate=\enumerate
\renewcommand{\enumerate}{\Enumerate%
\setlength{\@topsep}{0pt}
\setlength{\itemsep}{0pt}%
\setlength{\parskip}{0pt plus 1pt}%
\renewcommand{\theenumi}{\textup{(\alph{enumi})}}%
\renewcommand{\labelenumi}{\theenumi}%
}
\let\endEnumerate=\endenumerate
\renewcommand{\endenumerate}{\endEnumerate\unskip}
\makeatother
\makeatletter
\def\@seccntformat#1{\csname the#1\endcsname.\quad}
\makeatother
\newcommand{\authortitle}[2]{\author{#1}\title{#2}\markboth{#1}{#2}}
\newcommand{\art}[6]{{\sc #1, \rm #2, \it #3 \bf #4 \rm (#5), \mbox{#6}.}}
\newcommand{\auth}[2]{{#1, #2.}}

\newcommand{\artprep}[3]{{\sc #1, \rm #2, #3.}}
\newcommand{\arttoappear}[3]{{\sc #1, \rm #2, to appear in \it #3}}
\newcommand{\book}[3]{{\sc #1, \it #2, \rm #3.}}
\newcommand{\AND}{{\rm and }}
\RequirePackage{amsthm}
\newtheoremstyle{descriptive}%
  {\topsep}   
  {\topsep}   
  {\rmfamily} 
  {}          
  {\bfseries} 
  {.}         
  { }         
  {}          
\newtheoremstyle{propositional}%
  {\topsep}   
  {\topsep}   
  {\itshape}  
  {}          
  {\bfseries} 
  {.}         
  { }         
  {}          
\theoremstyle{propositional}
\newtheorem{thm}{Theorem}[section]
\newtheorem{prop}[thm]{Proposition}
\newtheorem{lem}[thm]{Lemma}
\theoremstyle{descriptive}
\newtheorem{deff}[thm]{Definition}
\newtheorem{example}[thm]{Example}
\makeatletter
\renewenvironment{proof}[1][\proofname]{\par
  \pushQED{\qed}%
  \normalfont
  \trivlist
  \item[\hskip\labelsep
        \itshape
    #1\@addpunct{.}]\ignorespaces
}{%
  \popQED\endtrivlist\@endpefalse
}
\makeatother
\newcommand{\setm}{\setminus}
\renewcommand{\subsetneq}{\varsubsetneq}
\renewcommand{\emptyset}{\varnothing}
\newcommand{\Cp}{{C_p}}
\newcommand{\Cn}{{C_n}}
\newcommand{\Cpmu}{{C_{p,\mu}}}
\DeclareMathOperator{\Div}{div}
\DeclareMathOperator{\capp}{cap}
\newcommand{\cpmu}{\capp_{p,\mu}}
\DeclareMathOperator{\dist}{dist}
\DeclareMathOperator*{\essliminf}{ess\,lim\,inf}
\DeclareMathOperator{\interior}{int}
\newcommand{\bdy}{\partial}
\newcommand{\loc}{_{\rm loc}}
{\catcode`p =12 \catcode`t =12 \gdef\eeaa#1pt{#1}}      
\def\accentadjtext#1{\setbox0\hbox{$#1$}\kern   
                \expandafter\eeaa\the\fontdimen1\textfont1 \ht0 }
\def\accentadjscript#1{\setbox0\hbox{$#1$}\kern 
                \expandafter\eeaa\the\fontdimen1\scriptfont1 \ht0 }
\def\accentadjscriptscript#1{\setbox0\hbox{$#1$}\kern   
                \expandafter\eeaa\the\fontdimen1\scriptscriptfont1 \ht0 }
\def\accentadjtextback#1{\setbox0\hbox{$#1$}\kern       
                -\expandafter\eeaa\the\fontdimen1\textfont1 \ht0 }
\def\accentadjscriptback#1{\setbox0\hbox{$#1$}\kern     
                -\expandafter\eeaa\the\fontdimen1\scriptfont1 \ht0 }
\def\accentadjscriptscriptback#1{\setbox0\hbox{$#1$}\kern 
                -\expandafter\eeaa\the\fontdimen1\scriptscriptfont1 \ht0 }
\def\itoverline#1{{\mathsurround0pt\mathchoice
        {\rlap{$\accentadjtext{\displaystyle #1}
                \accentadjtext{\vrule height1.593pt}
                \overline{\phantom{\displaystyle #1}
                \accentadjtextback{\displaystyle #1}}$}{#1}}
        {\rlap{$\accentadjtext{\textstyle #1}
                \accentadjtext{\vrule height1.593pt}
                \overline{\phantom{\textstyle #1}
                \accentadjtextback{\textstyle #1}}$}{#1}}
        {\rlap{$\accentadjscript{\scriptstyle #1}
                \accentadjscript{\vrule height1.593pt}
                \overline{\phantom{\scriptstyle #1}
                \accentadjscriptback{\scriptstyle #1}}$}{#1}}
        {\rlap{$\accentadjscriptscript{\scriptscriptstyle #1}
                \accentadjscriptscript{\vrule height1.593pt}
                \overline{\phantom{\scriptscriptstyle #1}
                \accentadjscriptscriptback{\scriptscriptstyle #1}}$}{#1}}}}
\newcommand{\limplus}{{\mathchoice{\vcenter{\hbox{$\scriptstyle +$}}}
  {\vcenter{\hbox{$\scriptstyle +$}}}
  {\vcenter{\hbox{$\scriptscriptstyle +$}}}
  {\vcenter{\hbox{$\scriptscriptstyle +$}}}
}}
\newcommand{\limminus}{{\mathchoice{\vcenter{\hbox{$\scriptstyle -$}}}
  {\vcenter{\hbox{$\scriptstyle -$}}}
  {\vcenter{\hbox{$\scriptscriptstyle -$}}}
  {\vcenter{\hbox{$\scriptscriptstyle -$}}}
}}
\newcommand{\alp}{\alpha}
\newcommand{\be}{\beta}
\newcommand{\Om}{\Omega}
\renewcommand{\phi}{\varphi}
\newcommand{\eps}{\varepsilon}
\newcommand{\z}{\zeta}
\newcommand{\p}{{$p\mspace{1mu}$}}
\newcommand{\R}{\mathbf{R}}
\newcommand{\Hp}{H^{1,p}}
\newcommand{\Hploc}{H^{1,p}\loc}
\newcommand{\Hpind}[1]{{H}_{#1}}      
\newcommand{\uHpind}[1]{\itoverline{H}_{#1}}      
\newcommand{\A}{\ensuremath{\mathcal{A}}}%
\newcommand{\Omt}{{\widetilde{\Omega}}}
\newcommand{\Gt}{\widetilde{G}}
\newcommand{\Ut}{\widetilde{U}}
\newcommand{\clG}{\itoverline{G}}     
\newcommand{\ut}{\tilde{u}}
\newcommand{\vt}{\tilde{v}}
\newcommand{\wt}{\widetilde{w}}
\numberwithin{equation}{section}
\newcommand{\eqv}{\ensuremath{
         \mathchoice{\quad \Longleftrightarrow \quad}{\Leftrightarrow}
                {\Leftrightarrow}{\Leftrightarrow}}}
\newcommand{\imp}{\ensuremath{\mathchoice{\quad \Longrightarrow \quad}{\Rightarrow}
                {\Rightarrow}{\Rightarrow}}}
\newenvironment{ack}{\medskip{\it Acknowledgement.}}{}

\begin{document}

\authortitle{Anders Bj\"orn}
{Removable singularities for bounded \A-(super)harmonic functions
 on weighted $\R^n$}
\title
{Removable singularities for bounded \A-(super)harmonic 
and   quasi(super)harmonic  functions
  on weighted $\R^n$}

\author{
Anders Bj\"orn \\
\it\small Department of Mathematics, Link\"oping University, SE-581 83 Link\"oping, Sweden\\
\it \small anders.bjorn@liu.se, ORCID\/\textup{:} 0000-0002-9677-8321
}

\date{}
\maketitle

\noindent{\small {\bf Abstract}.  
It is well known that sets of \p-capacity zero
are removable for bounded \p-harmonic functions,
but on metric spaces there are 
examples of removable sets of positive capacity.
In this paper, we show that this can happen even on unweighted $\R^n$
when  $n > p$, although only in very special cases.
A complete characterization of removable
singularities for bounded \A-harmonic 
functions on weighted $\R^n$, $n \ge 1$, is also given,
where the weight is \p-admissible.
The same characterization is also shown to hold
for bounded quasiharmonic functions
on weighted $\R^n$, $n \ge 2$, as well as on unweighted $\R$.
For bounded \A-superharmonic functions and bounded quasisuperharmonic functions
on weighted $\R^n$, $n \ge 2$, 
we show that relatively closed sets are removable if and only if
they have zero capacity.
}

\medskip

\noindent
{\small \emph{Key words and phrases}:
\A-harmonic function,
\A-superharmonic function,
bounded,
\p-admissible weight,
\p-harmonic function,
\p-superharmonic function,
quasiharmonic function,
quasisuperharmonic function,
removable singularity,  
weighted Euclidean space.
}

\medskip
\noindent
{\small Mathematics Subject Classification (2020):
Primary:
31C45, 
Secondary:
31E05, 
35J92. 
}

\section{Introduction}

Sets of \p-capacity zero are removable for bounded
\A-harmonic functions, more precisely:

\begin{thm} \label{thm-1.1}
If $\Om \subset \R^n$ is open
and $E \subsetneq \Om$ is a relatively closed set with 
\p-capacity $\Cp(E)=0$, then every bounded \A-harmonic function
on $\Om \setm E$ has an \A-harmonic extension to $\Om$.
Moreover, the extension is unique and bounded.
\end{thm}

Here,  and throughout the paper, $1<p<\infty$
and
an \emph{\A-harmonic function} is a continuous
weak solution of the \A-harmonic equation
\[
       \Div \A(x,\nabla u)=0,
\]
where
$\A$ satisfies the degenerate ellipticity
conditions (3.3)--(3.7)
in \cite[p.~56]{HeKiMa} (with the parameter $p$).
The \emph{\p-harmonic functions}, which are the continuous
weak solutions of the \p-Laplace equation
$     \Div(|\nabla u|^{p-2}\nabla u)=0$,
are included as a special case.

Theorem~\ref{thm-1.1} 
is well known and goes back to Serrin~\cite[Theorem~10]{serrin64},
who covered the case when $E$ is compact.
It has later been extended to
quasiharmonic functions on unweighted $\R^n$ (with $E$ compact)
in Tolksdorff~\cite[Theorem~1.2]{tolksdorf},
\A-harmonic functions on
weighted $\R^n$ in Heinonen--Kilpel\"ainen--Martio~\cite[Theorem~7.36]{HeKiMa},
and \p-harmonic and quasiharmonic functions on metric spaces in
Bj\"orn--Bj\"orn--Shanmugalingam~\cite[Proposition~8.2]{BBS2}
and Bj\"orn~\cite[Theorem~6.2]{ABremove},
covering also relatively open $E$.
(A \emph{quasiharmonic function} is a continuous quasiminimizer.)

A natural question is if the converse is true:
\emph{If $E$ as above is removable for bounded \p-harmonic or bounded \A-harmonic
functions in $\Om \setm E$, does it follow that $\Cp(E)=0$?}
(To avoid pathological cases, we
assume that no component of $\Om$ is contained in $E$.)

Maz$'$ya~\cite[Remark~1.4]{mazya72} showed this converse
for compact $E$ when $p<n$.
For compact $E$ and bounded $\Om$,
this was extended to
unweighted and weighted $\R^n$, $1<p<\infty$, by
Heinonen--Kilpel\"ainen~\cite[Remark~4.8]{HeKi88}
and Heinonen--Kilpel\"ainen--Martio~\cite[comment after Theorem~7.36]{HeKiMa},
respectively.
For \p-harmonic functions on metric spaces $X$ it
is due to Bj\"orn--Bj\"orn--Shanmugalingam~\cite[Proposition~8.4]{BBS2},
when $E$ is compact, $\Om$ is bounded
and the capacity $\Cp(X \setm \Om)>0$.

On the other hand, in \cite[Example~9.3]{ABremove}, the present author
gave examples,
for the metric space $[0,1]$, 
of compact sets with positive capacity (and even with positive measure)
removable for
bounded \p-harmonic and quasiharmonic functions.
In particular this shows that the condition $\Cp(X \setm \Om)>0$
above is essential.
Now we are able to show that this can happen even on unweighted
$\R^n$, $n \ge 2$, but only in very special cases,
namely when all the involved functions are constant.

Assume from now on that $\Om \subset \R^n$ is nonempty and open,
and that $E \subsetneq \Om$ is a relatively closed subset
such that no component of $\Om$ is contained in $E$.
(The last condition is to avoid noninteresting pathological situations.)

For unweighted $\R^n$, $n \ge 2$,
we have the following complete characterization
of removable sets for bounded \A-harmonic functions.

\begin{thm} \label{thm-main-intro-unweighted}
Assume that $\Om \subset \R^n$ \textup{(}unweighted\/\textup{)} with $n \ge 2$.
Then $E$ is removable for bounded \A-harmonic functions in $\Om \setm E$
if and only if one of the following two disjoint cases is true\/\textup{:}
\begin{enumerate}
\item \label{a-a}
$\Cp(E)=0$\textup{;}
\item \label{a-b}
$p > n$, $\Om=\R^n$ and $E$ is a singleton set.
\end{enumerate}
\end{thm}

Note that when $p>n$ all points have positive $\Cp$-capacity,
and thus \ref{a-b} gives examples of removable sets with positive capacity,
although 
only 
when all bounded \A-harmonic functions on $\Om \setm E$
are constant.

On weighted  $\R^n$ 
equipped with a \p-admissible weight $w$,
we are also able to give a complete characterization,
albeit 
a bit more involved.

\begin{thm} \label{thm-main-Rn}
Assume that $\Om \subset \R^n$ with $n \ge 2$,
where $\R^n$ is 
equipped with $d\mu=w\,dx$ and $w$ is a \p-admissible weight.
Then the following are equivalent\/\textup{:}
\begin{enumerate}
\item \label{r-pharm}
$E$ is removable for bounded \A-harmonic functions in $\Om \setm E$\textup{;}
\item \label{r-qharm}
$E$ is removable for bounded quasiharmonic functions in $\Om \setm E$\textup{;}
\item \label{r-b}
  either $\Cpmu(E)=0$, or
  there is some $x_0 \in E$ such that 
$\Cpmu(E \setm \{x_0\})=\Cpmu(\R^n \setm \Om)=0$
and $(\R^n,\mu)$ is \p-parabolic.
\end{enumerate}

Moreover, in \ref{r-qharm} the quasiharmonicity constant $Q$ is
preserved.
\end{thm}

In both theorems above, $\A$ is assumed to be fixed,
but by the characterizations removability for one $\A$ holds if and only if it
holds for all other \A, as long as $p$ and $\mu$ are fixed.
Since \ref{r-pharm} $\eqv$ \ref{r-qharm} in Theorem~\ref{thm-main-Rn},
it follows that the conditions in
Theorem~\ref{thm-main-intro-unweighted}
also characterize removability for bounded quasiharmonic functions
on unweighted $\R^n$.

Naively it may seem that once nonremovability has been shown
for compact sets with positive capacity
(disregarding for the moment the exceptional cases above)
it would follow for relatively closed sets $E$ with positive capacity,
as $E$ contains a compact set with positive capacity.
However, this is a little more subtle, and indeed on $\R$ there are
examples of relatively closed removable sets with compact nonremovable
subsets (for bounded \p-harmonic functions),
see Theorem~\ref{thm-R} below or \cite[Example~9.1]{ABremove}.
Ruling out such possibilities on $\R^n$, $n \ge 2$,
is therefore an important
part of the proofs of Theorems~\ref{thm-main-intro-unweighted} and~\ref{thm-main-Rn}.
This may also be the reason why 
nonremovability results for noncompact sets first seems to appear in 
Bj\"orn--Bj\"orn--Shanmugalingam~\cite[Propositions~8.4 and~8.5]{BBS2}.

Compared with the situation in (weighted) $\R^n$, $n \ge 2$, as depicted
above, the situation can be quite different in metric spaces.
Even on the unweighted real line $\R$ the situation is fundamentally
different than in higher-dimensional Euclidean spaces, see
\cite[Example~9.1]{ABremove}.
To get a feeling for this, we include this case in our
studies here.
We are now able to give the following complete
characterization of removable singularities for
bounded \A-harmonic functions on weighted $\R$.

\begin{thm} \label{thm-R}
  The set $E$ is removable for bounded \A-harmonic  functions
  in $\Om \setm E$, with respect to $(\R,w)$,
where $w$ is a \p-admissible weight,
  if and only if
  there is a constant $C$ such that
  for every component $I$ of $\Om$ it is
  true that $I \setm E$ is connected and 
\begin{equation} \label{eq-I}
  |I| \le C |I \setm E|,
\end{equation}
where $|\cdot|$ denotes the Lebesgue measure.
Moreover, in this case the extensions are unique.
\end{thm}  

Note that this time the removability condition is 
not only independent of $\A$, but also of $w$ and $p$.
In Theorem~\ref{thm-R-qharm} we show that removable
for bounded quasiharmonic functions on \emph{unweighted} $\R$
is characterized by the same condition.

A related topic is 
removability for bounded  \A-superharmonic
and bounded quasisuperharmonic functions.
Also in this case,
it has been shown that sets of zero capacity are removable, 
see Tolksdorff~\cite[Theorem~1.5]{tolksdorf},
Heinonen--Kil\-pe\-l\"ai\-nen~\cite[Theorem~4.7]{HeKi88},
Heinonen--Kilpel\"ainen--Martio~\cite[Theorem~7.35]{HeKiMa},
Bj\"orn--Bj\"orn--Shan\-mu\-ga\-lin\-gam~\cite[Proposition~8.3]{BBS2}
and Bj\"orn~\cite[Theorem~6.3]{ABremove}.

For compact $E$ with positive capacity and bounded $\Om$
several of the nonremovability results mentioned above also
apply to \A-superharmonic and quasisuperharmonic functions.
More specifically, nonremovability  was obtained for 
\A-superharmonic functions on unweighted and weighed $\R^n$ by
Heinonen--Kil\-pe\-l\"ai\-nen~\cite[Remark~4.8]{HeKi88}
and Heinonen--Kilpel\"ainen--Martio~\cite[comment after Theorem~7.36]{HeKiMa},
respectively,
while nonremovability for \p-superharmonic functions and quasisuperharmonic
functions on metric spaces $X$
(under the additional condition that $\Cp(X \setm \Om)>0$)
is due to Bj\"orn--Bj\"orn--Shanmugalingam~\cite[Proposition~8.4]{BBS2}
and Bj\"orn~\cite[Proposition~7.2]{ABremove}, respectively.

We can now show that on weighted $\R^n$, $n \ge2$,
relatively closed sets are removable for bounded \A-superharmonic
and bounded quasisuperharmonic functions
if and only if they have capacity zero, with no exceptional cases.
More precisely we obtain the following result.

\begin{thm} \label{thm-superh}
Assume that $\Om \subset \R^n$ with $n \ge 2$,
where $\R^n$ is 
equipped with $d\mu=w\,dx$ and $w$ is a \p-admissible weight.
Then the following are equivalent\/\textup{:}
\begin{enumerate}
\item \label{s-pharm}
$E$ is removable for bounded \A-superharmonic functions in $\Om \setm E$\textup{;}
\item \label{s-qharm}
$E$ is removable for bounded quasisuperharmonic functions in $\Om \setm E$\textup{;}
\item \label{s-pharm-2}
  $E$ is removable for \A-superharmonic functions in $\Om \setm E$
  which are bounded from below\textup{;}
\item \label{s-qharm-2}
  $E$ is removable for quasisuperharmonic functions in $\Om \setm E$
    which are bounded from below\textup{;}
\item \label{s-a}
$\Cpmu(E)=0$.
\end{enumerate}

Moreover, in~\ref{s-qharm} and~\ref{s-qharm-2}
the quasisuperharmonicity constant $Q$ is
preserved.
\end{thm}

In fact, when $\Cpmu(E)>0$ we construct a bounded
\A-superharmonic function on $\Om \setm E$
which has no quasisuperharmonic extension to $\Om$ (neither bounded nor unbounded).

A \emph{quasiharmonic function} is a continuous quasiminimizer.
Quasiminimizers were introduced by
Giaquinta and Giusti~\cite{GG1}, \cite{GG2} as a tool for a unified
treatment of variational integrals, elliptic equations and
quasiregular mappings on $\R^n$.  
They showed that De Giorgi's
method could be extended to quasiminimizers and obtained, in particular,
their local H\"older continuity. 
Thus every quasiminimizer has a continuous representative,
and it is this representative that is \emph{quasiharmonic}.
Similarly, every quasisuperminimizer has a unique
quasisuperharmonic representative 
(just as \p-supersolutions have unique \p-superharmonic representatives). 

Kinnunen--Martio~\cite{KiMa03} showed that one can
build a potential theory on quasiharmonic functions.
In particular, they introduced the quasisuperharmonic 
functions, and studied their potential theory on metric spaces.
Martio--Sbordone~\cite{MaSb} showed that quasiminimizers
have an interesting theory also in the one-dimensional case
on unweighted $\R$. 
Our Theorem~\ref{thm-R-qharm} is a contribution to this case.
It seems as if all papers (so far) studying quasiminimizers
and/or quasi(super)harmonic functions on weighted $\R^n$, $n \ge 2$, do so as
a particular case of the corresponding theory on metric spaces
(cf.\ however \cite[Section~3.13]{HeKiMa}).
See e.g.\ the introduction in Bj\"orn--Bj\"orn--Korte~\cite{BBKorte}
for further references on quasiminimizers.

Various other types of
removable singularities for \p-(super)harmonic, \A-(super)\-harmonic
and quasi(super)harmonic have been studied in
Heinonen--Kil\-pe\-l\"ai\-nen~\cite{HeKi88},
Kilpel\"ainen--Zhong~\cite{KilpZhong02},
Pokrovskii~\cite{pokrovskii},
M\"ak\"al\"ainen~\cite{makalainen}
and   Bj\"orn~\cite{ABremove}.
For harmonic and analytic functions removable singularities
lying in many different types of function spaces have been studied,
but this is not the right place to discuss this.

\begin{ack}
The  author was supported by the Swedish Research Council,
grant 2016-03424.
The author is grateful to Juha Heinonen, Tero Kilpel\"ainen,
Nageswari Shanmugalingam and Jana Bj\"orn for fruitful discussions.
\end{ack}

\section{Notation and preliminaries}
\label{sect-prelim}

\emph{Throughout the paper we let $1<p< \infty$, equip $\R^n$ with
$d\mu=w\,dx$, where $w$ is a \p-admissible weight,
and assume that 
$\A$ satisfies the degenerate ellipticity
conditions\/ \textup{(3.3)--(3.7)}
in\/ \textup{\cite[p.~56]{HeKiMa}} \textup{(}with the parameter $p$\textup{)}.
We also let $\Om \subset \R^n$ be a nonempty open set
and let $E \subsetneq \Om$ be a relatively closed subset
such that no component of $\Om$ is contained in $E$.}

\medskip

We will follow the notation in
Heinonen--Kilpel\"ainen--Martio~\cite{HeKiMa}.
In particular, \p-admissible weights 
are defined in Section~1.1 in \cite{HeKiMa}.
See also 
Corollary~20.9 in~\cite{HeKiMa}
and Proposition~A.17 in~\cite{BBbook}
for characterizations of \p-admissible weights.
We refer to 
Chapters~6 and~7 in~\cite{HeKiMa} for
the definitions of \A-harmonic and \A-superharmonic functions.

\begin{deff} \label{def-quasimin}
Let $Q \ge 1$.
A function $u \in \Hploc(\Om,\mu)$ is a
\emph{$Q$-quasisuperminimizer} in $\Om$
if 
\begin{equation} \label{eq-qmin}
      \int_{\phi \ne 0} |\nabla u|^p \, d\mu
           \le Q\int_{\phi \ne 0} |\nabla (u+\phi)|^p \, d\mu
           \quad \text{for all nonnegative } \phi \in \Hp_0(\Om,\mu).
\end{equation}
If \eqref{eq-qmin} holds for all $\phi \in \Hp_0(\Om,\mu)$,
then $u$ is a \emph{$Q$-quasiminimizer}.
A \emph{$Q$-quasi\-har\-mon\-ic} function is a continuous $Q$-quasiminimizer.
\end{deff}

These notions depend on $p$ and $\mu$, but we have refrained
from making that explicit in the notation.

Kinnunen--Martio~\cite{KiMa03} introduced quasisuperharmonic
functions and studied their potential theory.
For us the following definition will be convenient,
see \cite{KiMa03}, \cite{ABkellogg} and \cite{ABremove} for equivalent
definitions and characterizations
of quasi(super)minimizers and quasi(super)harmonic functions.

\begin{deff} \label{def-qsh}
A function $u : \Om \to (-\infty,\infty]$
is  \emph{$Q$-quasisuperharmonic} in\/ $\Om$
if $u$ is not identically\/ $\infty$ in any component of\/ $\Om$,
$u$ is \emph{essliminf-regularized} i.e.
\[ 
  u(x)=\essliminf_{y \to x} u(y) \quad \text{for }x \in \Om,
\] 
and\/ $\min\{u,k\}$ is a $Q$-quasisuperminimizer in\/ $\Om$ for
every $k \in \R$.
\end{deff}

\A-harmonic  and \A-superharmonic functions
are $Q$-quasiharmonic resp.\ $Q$-quasi\-super\-harmonic,
with $Q$ only depending on \A,
see \cite[Section~3.13]{HeKiMa},
whereas \p-(super)harmonic functions are the same
as $1$-quasi(super)harmonic functions.

We next collect the
key removability results for \A-harmonic and quasiharmonic functions,
see \cite[Theorem~7.36]{HeKiMa} and \cite[Theorem~6.2]{ABremove}.

\begin{thm} \label{thm-removability-qharm}
Assume that the Sobolev capacity $\Cpmu(E)=0$.
Then $E$ is removable for \A-harmonic and $Q$-quasiharmonic functions
in $\Om \setm E$,
i.e.\ 
every bounded \A-harmonic\/ \textup{(}resp.\ $Q$-quasiharmonic\/\textup{)}
function in\/ $\Om \setm E$
has a bounded \A-harmonic\/ \textup{(}resp.\ $Q$-quasiharmonic\/\textup{)}
extension to $\Om$.
\end{thm}

The following Harnack inequality
for quasiharmonic functions is important.
In the weighted setting it was first
obtained
by Kinnunen--Shan\-mu\-ga\-lin\-gam~\cite[Corollary~7.3]{KiSh01}
(who obtained it in metric spaces).

\begin{thm} \textup{(The Harnack inequality)}
\label{thm-Harnack}
There exists a constant $C$, depending only on
$Q$, $p$ and $\mu$,
such that for all 
$Q$-quasiharmonic functions $u  \ge 0$ in $\Om$ and all balls $B$ with
$6B\subset\Om$, 
\[ 
	\sup_B u \le C \inf_B u.
\] 
\end{thm}

The following are two important 
consequences of
the Harnack inequality.

\begin{thm} \label{thm-max}
\textup{(The strong maximum principle)}
  A quasiharmonic function that attains
  its maximum in a connected open set $\Om$ must be constant in $\Om$.
\end{thm}

\begin{thm} \label{thm-Liouville}
\textup{(The Liouville theorem)}
  A quasiharmonic function in $\Om=\R^n$ which is bounded
  from below\/ \textup{(}or from above\/\textup{)} must be constant.
\end{thm}

\begin{deff}
The space $(\R^n,\mu)$ is \emph{\p-parabolic} if
the variational capacity
\[
\cpmu(B,\R^n)=0
\]
for some ball $B$,
or equivalently for all balls $B$.

If $(\R^n,\mu)$ is not \p-parabolic then it is called \emph{\p-hyperbolic}.
\end{deff}

It is easy to see that unweighted $\R^n$ is \p-parabolic if and only if $p \ge n$.
See 
\cite[Theorem~9.22]{HeKiMa}
and Proposition~\ref{prop-char-reg} below
for various characterizations of \p-parabolicity.
Bj\"orn--Bj\"orn--Lehrb\"ack~\cite[Theorem~5.5]{BBLehIntGreen}
recently showed that $(\R^n,\mu)$ is \p-parabolic if and only if
\begin{equation}   \label{eq-int-parab=infty}
\int_{1}^\infty \biggl( \frac{r}{\mu(B(0,r))} \biggr)^{1/(p-1)} \,dr=\infty.
\end{equation}
They obtained this characterization in metric spaces under
very mild assumptions.

\section{Weighted
  \texorpdfstring{$\R^n$}{Rn} with \texorpdfstring{$n\ge2$}{n>=2};
Theorems~\ref{thm-main-intro-unweighted} and~\ref{thm-main-Rn}}
\label{sect-Rn}

\emph{In addition to the general assumptions from the beginning
  of Section~\ref{sect-prelim}, throughout this section we assume that $n \ge 2$.}

\medskip

Our aim in this section is
to prove Theorem~\ref{thm-main-Rn}.
(Theorem~\ref{thm-main-intro-unweighted}
will then be  a more or less direct consequence.)
To do so we need some lemmas.

\begin{lem} \label{lem-x->infty}
Let $F \subset \R^n$
 be a bounded closed set and $\Om=\R^n \setm F$.
Also let $u$ be a bounded  quasiharmonic function in $\Om$.
Then $\lim_{x \to \infty} u(x)$ exists.
\end{lem}

The idea for the proof is to use Harnack's inequality on spheres.
To make the proof a little more elementary we will only use the
Harnack inequality (Theorem~\ref{thm-Harnack}) above.

\begin{proof}
Without loss of generality we may assume that
$F=\itoverline{B(0,R)}$.
Let 
$m(r)=\min_{x \in \bdy B(0,r)} u(x)$, $r>R$.
As $u$ is continuous, so is $m$.
Moreover, by the strong maximum principle
(Theorem~\ref{thm-max}),
$m$ does not have any local minimum
(unless $u$ is constant).
In particular, 
 $M:=\lim_{r \to \infty} m(r)$ exists (also if $u$ is constant).
We may assume that $M=0$.
It follows that $\liminf_{x \to \infty} u =0$.

Let $\eps>0$ and find $\rho > R$ such that
$\sup_{r \ge \rho} |m(\rho)|<\eps$.
Assume that $|x|>2\rho$.
Find $z \in \bdy B(0,\tau)$ such that $u(z)=m(\tau)$, where $\tau=|x|$.
Then we can find a sequence of balls
$B_j=B\bigl(x_j,\tfrac{1}{6}\tau)$, $j=0,\ldots,N$,
where $x_0=x$, $x_N=z$ and $x_j \in B_{j-1} \cap \bdy B(0,\tau)$,
$j=1,\ldots,N$.
Moreover $N$ can be chosen to only depend on the dimension $n$.
By applying the Harnack inequality (Theorem~\ref{thm-Harnack})
$N$ times to $v=u+\eps$, we see that
\[
    v(x)=v(x_0) \le C v(x_1) \le \ldots \le C^N v(x_N) = C^N v(z) \le 2C^N\eps.
\]
Hence
\[
    \sup_{|x|>2\rho} u(x) <\sup_{|x|>2\rho} v(x) \le  2C^{N}\eps.
\]
Letting $\eps \to 0$ shows that $\limsup_{x \to \infty} u \le 0$.
\end{proof}

\begin{lem} \label{lem-x->x0}
  Let  $x_0 \in \Om$
  and let $u$ be a bounded  quasiharmonic function in $\Om \setm \{x_0\}$.
Then $\lim_{x \to x_0} u(x)$ exists.
\end{lem}

\begin{proof}
The proof is essentially the same as the proof
of Lemma~\ref{lem-x->infty}.
\end{proof}

\begin{prop} \label{prop-char-reg}
  Let $x_0 \in \R^n$ be such that $\Cpmu(\{x_0\})>0$.
Then the following are equivalent\/\textup{:}
\begin{enumerate}
\item \label{i-p-hyp}
$(\R^n,\mu)$ is \p-hyperbolic\textup{;} 
\item \label{i-p-x0}
$\infty$ is regular with respect to $\R^n \setm \{x_0\}$\textup{;}
\item \label{i-p-pharm}
there is a bounded nonconstant \A-harmonic
function in $\R^n \setm \{x_0\}$\textup{;}
\item \label{i-p-qharm}
there is a bounded nonconstant quasiharmonic function 
in $\R^n \setm \{x_0\}$\textup{;}
\end{enumerate}
\end{prop}

There are further characterizations in Theorem~9.22
in Heinonen--Kilpel\"ainen--Martio~\cite{HeKiMa};
see also \eqref{eq-int-parab=infty}.

In the proof we will use \A-harmonic Perron solutions $\Hpind{\Om} f$ and boundary
regularity, see Chapter~9 in~\cite{HeKiMa}.
When we discuss boundary regularity in this paper it is always with respect
to \A-harmonic functions.
(In this context, the dependence on $p$ and $\mu$ will be implicit.)
In particular, we will use the Kellogg property which says
that $\Cpmu(I)=0$, where $I$ is the set of finite irregular boundary points
of an open set $\Om$, see Theorem~9.11 in~\cite{HeKiMa}.

\begin{proof}
\ref{i-p-hyp} $\imp$ \ref{i-p-x0}
This follows from Theorem~9.22 in~\cite{HeKiMa}.

\ref{i-p-x0} $\imp$ \ref{i-p-pharm} 
Let $f(x_0)=1$ and $f(\infty)=0$.
Then
$u=\Hpind{\R^n \setm \{x_0\}} f$ is \A-harmonic.
As $x_0$ is regular, by the Kellogg property,
we have $\lim_{y \to x_0} u(y)=1$.
Moreover, as $\infty$ is regular, by assumption, $\lim_{x \to \infty} u(x) =0$.
Thus 
$u$ is a bounded nonconstant \A-harmonic
function in $\R^n \setm \{x_0\}$.

\ref{i-p-pharm} $\imp$ \ref{i-p-qharm}
This is trivial.

\ref{i-p-qharm} $\imp$ \ref{i-p-hyp}
Let $u$ be a bounded nonconstant $Q$-quasiharmonic function 
in $\R^n \setm \{x_0\}$.
By Lemmas~\ref{lem-x->infty} and~\ref{lem-x->x0},
both the limits $u(x_0):=\lim_{x \to x_0} u(x)$
and $u(\infty):=\lim_{x \to \infty} u(x)$ exist.
As $u$ is nonconstant, the strong maximum principle
(Theorem~\ref{thm-max})
shows
that $u(x_0) \ne u(\infty)$.
We may assume that $u(\infty)=0$ and $u(x_0)=1$.
Let $0 < \eps < \tfrac14$.
Then there is $0 < r<\eps $ so that
$M:=\max_{|x-x_0|=r} u(x) > 1-\eps$.

Assume next that $(\R^n,\mu)$ is \p-parabolic.
Then $\cpmu(\itoverline{B(x_0,r)},\R^n)=0$.
As the admissible functions testing the capacity
have compact support,  there is $R>1/\eps$ such that
$\cpmu(\itoverline{B(x_0,r)},B(x_0,R))< \eps$
and $m:=\min_{|x-x_0|=R} u(x) < \eps$.
Let $K =\{x:u(x) \ge M\}$ and $F=\{x \in \R^n:u(x) \le m\}$.
Then $K \subset \itoverline{B(x_0,r)}$ is compact and
$F \subset \R^n \setm B(x_0,R)$
is closed, by the continuity of $u$ and the strong
maximum principle (Theorem~\ref{thm-max}).

Let $v$ be the \p-harmonic function in
the bounded open set $\Om=\R^n \setm (K \cup F)$ with boundary
values $1$ on $K \cap \bdy \Om$ and $0$ on $F \cap \bdy \Om$
(in Sobolev or equivalently Perron sense, see
\cite[Corollary~9.29]{HeKiMa}).
Also let $\vt=(M-m)v+m$.
Then $\vt-u \in \Hp_0(\Om,\mu)$ and thus
\begin{align*}
\int_{\Om} |\nabla u|^p \, d\mu
& \le Q \int_{\Om} |\nabla \vt|^p \, d\mu
\le    Q\int_{\Om} |\nabla v|^p \, d\mu \\
&=   Q \cpmu(K,\R^n \setm F)
\le  Q \cpmu (\itoverline{B(x_0,r)},B(x_0,R))
<  Q\eps.
\end{align*}
Letting $\eps \to \infty$ shows that
\[
   \int_{\R^n \setm \{x_0\}} |\nabla u|^p \, d\mu =0,
\]
but this is impossible as $u$ is nonconstant.
Hence $(\R^n,\mu)$ must be \p-hyperbolic.
\end{proof}

\begin{lem} \label{lem-intE-nonempty}
If\/ $\interior E \ne \emptyset$, then $E$ is not
removable for bounded \A-harmonic functions in $\Om \setm E$,
nor for bounded quasiharmonic functions in $\Om \setm E$.
\end{lem}

This is a special case of Theorem~\ref{thm-main-Rn}, but it will convenient
to have it taken care of when we turn to the proof of Theorem~\ref{thm-main-Rn}
below.

In the first paragraph of the proof of Lemma~\ref{lem-intE-nonempty} below,
we will use a fact about separating sets in $\R^n$,
$n \ge 2$. In order to avoid sending the reader to ``topological'' papers we will
provide a potential-theoretic argument sufficient for our purpose.
A similar argument plays a role also in the second paragraph of the proof,
but then yielding a potential-theoretic consequence needed to conclude the proof.

\begin{proof}[Proof of Lemma~\ref{lem-intE-nonempty}]
Let $\Omt$ be a component of $\Om$ such that
$E':=\Omt \cap \interior E  \ne \emptyset$.
Then $E'':= \Omt \cap \bdy E'$ separates $\Omt$ into the two nonempty
disjoint open sets $E'$ and $\Omt \setm (E' \cup E'')$.
(The latter set is nonempty as 
$\Omt \setm (E' \cup E'') \supset \Omt \setm E \ne \emptyset$, by assumption.)
Since sets of $n$-capacity zero do not separate 
open connected sets in $\R^n$ (see \cite[Lemma~4.6]{BBbook}),
we must have $\Cn(E'')>0$.
Next, as singleton sets have $\Cn$-capacity zero, 
the set $E''$ must be uncountable, and since
it is relatively closed it must contain a limit point.
(The relatively closedness is not needed, but it is perhaps
easier to see this in this case.)
Hence there are $x_j\in E''$, $j=0,1,\ldots$\,,
such that
$|x_{j+1}-x_0| < d_j:=\frac{1}{4} |x_j-x_0|$,
$j=1,2,\ldots$\,.

Then the balls $B_j=B(x_j,d_j)$ are pairwise disjoint,
and $B_j$ contains points both in $G:=\Om \setm E$ and in
$\interior E \subset X \setm \clG$.
Let $F_j=B_j \setm G \subsetneq B_j$, which is a relatively closed
subset of the open set $B_j$.
Now $B_j \cap \bdy F_j$ separates $B_j$ into the two
nonempty open sets $B_j \cap G$ and $B_j \setm \clG$.
Since again sets of capacity zero cannot separate open connected sets,
we see that $\Cpmu(B_j \cap \bdy F_j)>0$.
As $B_j \cap \bdy F_j \subset \bdy G$, the Kellogg property shows 
that there is a point $y_j \in B_j \cap \bdy F_j$
that is regular with respect to $G$.
Let $f_j(x)=(1-|x-y_j|/d_j)_\limplus$.
Then $f_j$ have pairwise disjoint support.
Also let $f=\sum_{k=1}^\infty f_{2k}$ and $u=\uHpind{G} f$
(an \A-harmonic upper Perron solution).
By Lemma~9.6 in~\cite{HeKiMa},
\[
     \lim_{G \ni y \to y_j} u(y)=
	\begin{cases}
	1, & \text{if $j$ is even}, \\
	0, & \text{if $j$ is odd}.
	\end{cases}
\]
It follows that $\lim_{G \ni y \to x_0} u$ does not exist,
and thus $u$ does not have any continuous extension to $\Om$,
let alone any \A-harmonic or quasiharmonic extension.
\end{proof}

We are now ready to prove Theorem~\ref{thm-main-Rn}.

\begin{proof}[Proof of Theorem~\ref{thm-main-Rn}]
Assume first that \ref{r-b} holds.
If $\Cpmu(E)=0$, then both types of
removability follow directly from 
Theorem~\ref{thm-removability-qharm} (with the quasiharmonicity
constant $Q$ preserved).
Assume therefore that there is $x_0 \in E$ such that $\Cpmu(\{x_0\})>0$,
$\Cpmu(E \setm \{x_0\})=\Cpmu(\R^n \setm \Om)=0$
and $(\R^n,\mu)$ is \p-parabolic.

Let $u$ be a bounded \A-harmonic or bounded
quasiharmonic function on $\Om \setm E$.
Since $\Cpmu((E \setm \{x_0\})\cup (\R^n \setm \Om))=0$ and $\R^n \setm \{x_0\}$
is open, it follows from Theorem~\ref{thm-removability-qharm}
that $(E \setm \{x_0\})\cup (\R^n \setm \Om)$ is removable,
and hence that $u$ has a bounded \A-harmonic or
bounded quasiharmonic extension $U$ to $\R^n \setm \{x_0\}$.
As $(\R^n,\mu)$ is \p-parabolic, it follows from
Proposition~\ref{prop-char-reg} that $U$ is constant.
Thus also $u$ is constant and trivially has a bounded \A-harmonic (and quasiharmonic)
extension to $\Om$, i.e.\ $E$ is removable in both senses.

Conversely, assume that \ref{r-b} fails.
If $\interior E \ne \emptyset$,
then \ref{r-pharm} and \ref{r-qharm}
fail, by Lemma~\ref{lem-intE-nonempty},
so in the rest of the proof we can assume that $\interior E=\emptyset$.

By assumption $\Cpmu(E)>0$,
and thus there is a component $\Omt$ of $\Om$
such that $\Cpmu(\Omt \cap E)>0$.
If $\Gt:=\Omt \setm E$ is disconnected,
then $\chi_V$, where $V$ is a component of $\Gt$, is an \A-harmonic
function on $G$ with no continuous extension to $\Om$, let alone
any quasiharmonic extension, i.e.\ \ref{r-pharm} and \ref{r-qharm}
fail.
We can therefore assume that $E$ does not separate $\Omt$.

By the Kellogg property,
there is a point $x_1 \in \Omt \cap E$ that is regular with respect
$G:=\Om \setm E$.
It follows that $x_1 \in \bdy \Gt$.
We now consider three mutually disjoint cases.

\medskip

\emph{Case}~1.
\emph{$\Cpmu(\R^n \setm \Om)>0$.}

As $\interior E=  \emptyset$, we see that $\bdy \Om \subset \bdy G$.
And as sets of capacity zero cannot separate open sets
(see \cite[Lemma~4.6]{BBbook}),
we see that
$\Cpmu(\bdy \Gt \cap \bdy \Om) > 0$.
Hence, it follows from 
 the Kellogg property that there is another 
point $x_2 \in \bdy \Gt \cap \bdy \Om$ that
is regular with respect to $G$.
Let $f(x)=(1-|x-x_2|/|x_1-x_2|)_\limplus$ (with $f(\infty)=0$).
Then $u=\Hpind{G} f$ is an \A-harmonic function in $G$.
As $x_1$ and $x_2$ are regular, $u|_{\Gt}$ is nonconstant.

Assume that $u$ has a quasiharmonic extension $U$ to $\Om$.
By continuity and regularity (recall that $\interior E = \emptyset$)
we find that $U \ge 0$ and that $U(x_1)=0$.
As $\Omt$ is connected this contradicts the strong maximum
principle (Theorem~\ref{thm-max}),
showing
that $u$ cannot have an \A-harmonic or quasiharmonic extension $U$ to $\Om$,
i.e.\ both \ref{r-pharm} and \ref{r-qharm} fail.

\medskip

\emph{Case}~2.
\emph{$\Cpmu(\R^n \setm \Om)=0$ and
    $\Cpmu(E \setm \{x_1\})>0$.}

As sets of capacity zero cannot separate open sets,
$\Om$ must be connected and thus $\Gt=G$ and $E \subset \bdy G$.
Hence, by
 the Kellogg property there is another 
point $x_2 \in E \setm \{x_1\}$ that
is regular with respect to $G$.
We can now proceed exactly as in case~1.

\medskip

\emph{Case}~3.
\emph{$\Cpmu(E \setm \{x_1\})= \Cpmu(\R^n \setm \Om)=0$
  and $(\R^n,\mu)$ is \p-hyperbolic.}

By Proposition~\ref{prop-char-reg},
there is a nonconstant
bounded \A-harmonic function $\ut$ in $\R^n \setm \{x_1\}$.
As $\interior E=\emptyset$, it follows that
$u:=\ut|_{G}$ is a nonconstant bounded \A-harmonic function in $G$.
Assume that $u$ has
a bounded
\A-harmonic or bounded quasiharmonic
extension $U$ to $\Om$.
Since $\Cpmu(\R^n \setm \Om)=0$, it follows from
Theorem~\ref{thm-removability-qharm} that there
is a further quasiharmonic extension $\Ut$ to $\R^n$,
but this contradicts the Liouville theorem
(Theorem~\ref{thm-Liouville}).
\end{proof}

\begin{proof}[Proof of Theorem~\ref{thm-main-intro-unweighted}]
If $p \le n$, then all singletons have zero capacity,
and thus \ref{r-b} in Theorem~\ref{thm-main-Rn}
is equivalent to $\Cp(E)=0$.

If on the other hand $p>n$ then every singleton has positive
capacity and $\R^n$ is \p-parabolic,
and thus \ref{r-b} in Theorem~\ref{thm-main-Rn}
is equivalent to ``\ref{a-a} or \ref{a-b}'' in Theorem~\ref{thm-main-intro-unweighted}.

In both cases, the characterization now follows
directly from Theorem~\ref{thm-main-Rn}.
\end{proof}

The condition $\Cpmu(E \setm \{x_0\})=0$ in
Theorem~\ref{thm-main-Rn} can be rephrased
as saying that $\Cpmu$ restricted to $E$ is concentrated at (at most) one point.
This condition was characterized in Lemma~10.6 in
Bj\"orn--Bj\"orn--Shanmugalingam~\cite{BBSunifPI}.

\section{The one-dimensional real line \texorpdfstring{$\R$}{R};
Theorem~\ref{thm-R}}
\label{sect-R}

\emph{In addition to the general assumptions from the beginning
  of Section~\ref{sect-prelim}, throughout this section we assume that $n=1$.}

\medskip

As pointed out in
\cite[Example~9.1]{ABremove},
already on (unweighted) $\R$ the situation is very different
from $\R^n$, $n \ge2$.
In this case \p-harmonic functions are nothing but affine functions
$x \mapsto \alp x+\be$ 
(independent of $p$).
It is easy to see
that $E=[a,\infty)$ is removable for bounded \p-harmonic functions in $\R \setm E$,
for every $a \in \R$.
Moreover, if $-\infty < a < b<\infty$, then
every bounded \p-harmonic function $u$ on $(a,b)$
has a \p-harmonic extension to $\R$, although it is bounded only if $u$ is constant.

Let us now consider the weighted real line $(\R,\mu)$,
where as before 
$d\mu=w\,dx$ and $w$ is a \p-admissible weight,
It follows from
Bj\"orn--Buckley--Keith~\cite[Theorem~2]{BjBuKe}
that $w$ is a Muckenhoupt \emph{$A_p$-weight}, i.e.\ that
\begin{equation} \label{eq-Ap}
\sup_I \frac{1}{|I|} \int_I w\,dx
   \biggl(\frac{1}{|I|} \int_I w^{1/(1-p)}\,dx\biggr)^{p-1} < \infty,
\end{equation}
where the supremum is over all bounded open intervals $I$ and
$|\cdot|$ denotes the Lebesgue measure.
(We assume, by definition, that intervals are nonempty.)

An \A-harmonic function $u$ is in this case nothing but a
continuous weak
solution of
\[
    \frac{d}{dx} A(x)w(x) |u(x)|^{p-2}u(x) =0,
\]
where $A$ is a positive measurable function
a.e.-bounded away from $0$ and $\infty$.
Hence $\wt=A w$ is a.e.-comparable to $w$ and
is thus also an $A_p$-weight.
Moreover, $u$ is \A-harmonic with respect to $(\R,w)$ if and only if
it is \p-harmonic with respect to $(\R,\wt)$.

We can now give a complete characterization of removability for
\A-harmonic functions on $\R$.
We start with 
the following complete characterization
of weak removability, where by \emph{weakly removable}
we mean that the extensions are not required to be bounded.
Note that components of open subset of $\R$ are intervals.

\begin{prop} \label{prop-R-weak}
  The set $E$ is weakly removable for bounded \A-harmonic functions
  in $\Om \setm E$, with respect to $(\R,w)$,
  if and only if $I \setm E$ is connected for every component $I$ of $\Om$.

Moreover, in this case the extensions are unique.
\end{prop}  

First note that by Lemma~6.2 in 
Bj\"orn--Bj\"orn--Shanmugalingam~\cite{BBSliouville},
a function $u$ is \A-harmonic in an open interval $I \subset \R$
if and only if it is given by 
\begin{equation} \label{eq-a-int}
u(x)=b+a \int_0^x  \wt^{1/(1-p)} \, dt, \quad x \in I, 
\end{equation}
where $a,b \in \R$.

\begin{proof}
As for the sufficiency, let $u$ be an \A-harmonic function on $\Om \setm E$.
Let $I$ be a component of $\Om$.
As $E$ does not separate $I$,
using \eqref{eq-a-int} we can \A-harmonically extend $u|_{I \setm E}$  to $I$,
and this extension is unique.
Doing this in each component, we end up with an \A-harmonic extension of $u$ to $\Om$.

Conversely, if there is a component $I$ of $\Om$
such that $I \setm E$ is not connected, we let $V$ be a component
of $I \setm E$.
Then $u=\chi_V$
is an \A-harmonic function on $\Om \setm E$, which due to \eqref{eq-a-int}
does not have any \A-harmonic extension to $\Om$.
\end{proof}

We next turn to the proof of Theorem~\ref{thm-R}
(in which case the extensions \emph{are}
required to be bounded).

\begin{proof}[Proof of Theorem~\ref{thm-R}]
Let $d\nu = v\,dx$, where $v=\wt^{1/(1-p)}$.
Since $\wt$ is an $A_p$-weight, it follows directly
from the $A_p$-condition \eqref{eq-Ap} that $v$ is an $A_{p'}$-weight,
where $p'=p/(p-1)$ is the dual exponent.
In particular,
\begin{equation} \label{eq-bdd}
  \nu(I)<\infty
  \quad \text{if and only if}
  \quad I \text{ is bounded},
\end{equation}
whenever $I$ is an interval.

Assume that $I \setm E$ is connected for every component $I$ of $\Om$.
We first want to show that \eqref{eq-I} is equivalent to \eqref{eq-nu}
below.
By \cite[Lemma~15.5]{HeKiMa}
we see that
\eqref{eq-I} implies that
\[
\frac{\nu(I \setm E)}{\nu(I)} \ge c \biggl(\frac{|I \setm E|}{|I|}\biggr)^{p'}
\ge \frac{c}{C^{p'}}
\] 
if $I$ is a bounded component, 
and hence that
\begin{equation} \label{eq-nu}
  \nu(I) \le C' \nu(I \setm E).
\end{equation}
For unbounded $I$, \eqref{eq-nu} follows directly from \eqref{eq-I} and \eqref{eq-bdd},
since $I \setm E$ is connected.
The converse implication \eqref{eq-nu} $\imp$ \eqref{eq-I}
follows from Theorem~3 in Coifman--Fefferman~\cite{coifmanF}
or \cite[Lemma~15.8]{HeKiMa}.
(The implication \eqref{eq-I} $\imp$ \eqref{eq-nu} can also be obtained using
Coifman--Fefferman~\cite[Theorem~3 and Lemma~5]{coifmanF}.)

Next, we turn to 
the sufficiency,
and let $0 \le u \le 1$ be an \A-harmonic function on $\Om \setm E$.
Using \eqref{eq-a-int} in each component separately,
we can extend $u$ to $\Om$ as an \A-harmonic function.
It follows from \eqref{eq-a-int} and \eqref{eq-nu} that $|u| \le C'$.

Conversely, if there is some component $I$ of $\Om$ such that 
$I \setm E$ is disconnected, then $E$ is not weakly removable,
by Proposition~\ref{prop-R-weak}, and hence not removable.

Next, consider the case when $I \setm E$ is connected
for every component $I$, but there is no constant $C'$ as
in \eqref{eq-nu}.
We need to consider two cases.

\medskip

\emph{Case}~1.
\emph{There is a component $I$ of $\Om$ such that
$\nu(I \setm E) < \infty = \nu(I)$.}

In this case,
there is a nonconstant \A-harmonic function $u$ in $\Om \setm E$
which is identically $0$ outside $I$.
It has a unique \A-harmonic extension $U$ to $\Om$, by
Proposition~\ref{prop-R-weak}.
By \eqref{eq-a-int}, $u$ is bounded in $\Om \setm E$ but $U$ is
unbounded in $\Om$.

\medskip

\emph{Case}~2.
\emph{There is a sequence of components $I_j$ of $\Om$ such that
$\nu(I_j) > j \nu(I_j \setm E)$.}

In this case there is an \A-harmonic function $u$ in $\Om \setm E$
such that 
\[
   \inf_{I_j \setm E} u =0 \text{ and } 
   \sup_{I_j \setm E} u =1
   \quad \text{for } j=1,2,\ldots,
\]
and $u\equiv 0$ on $(\Om \setm \bigcup_{j=1}^\infty I_j) \setm E$. 
It has a unique \A-harmonic extension $U$ to $\Om$, by
Proposition~\ref{prop-R-weak}.
By \eqref{eq-a-int} we get that
\[
   \sup_{I_j} U- \inf_{I_j} U
   \ge \frac{\nu(I_j)} {\nu(I_j \setm E)}
   \Bigl(\sup_{I_j \setm E} u - \inf_{I_j \setm E} u\Bigr)
   >j,
\]
and thus $U$ is unbounded.
Hence $E$ is not removable.
\end{proof}

We now turn to quasiharmonic functions.
In the metric space Examples~9.3 and~9.4 in Bj\"orn~\cite{ABremove} it 
was shown that sets of positive measure can be removable for
bounded $Q$-quasiharmonic functions.
Using a result from
Bj\"orn--Bj\"orn--Shanmugalingam~\cite{BBSliouville},
we can now give a similar example on weighted $\R$.

\begin{example}
Let $w$ be a \p-admissible weight on $\R$, or equivalently
an $A_p$-weight (see above).
Let $\Om=\R$ and let $E \subsetneq \R$ be an unbounded closed
interval.
Then $I=\Om \setm E$ is an unbounded open interval.
By \eqref{eq-bdd}, $\nu(I)=\infty$, where $d\nu=w^{1/(1-p)} \,dx$.
It thus follows from  Proposition~6.6 in~\cite{BBSliouville} that
every bounded quasiharmonic function on $I$ is constant,
and hence trivially extends to $\R$.
Thus, we have shown  that $E$ is removable for bounded quasiharmonic
functions in $\R \setm E$, with respect to  $(\R,w)$.
Moreover, the quasiharmonicity constant (which is always $1$) is preserved.
\end{example}

In this example, removability depended on 
the fact that
all the quasiharmonic functions under consideration were constant.
The same is true for the removable sets
with positive capacity provided by
Examples~9.3 and~9.4 in Bj\"orn~\cite{ABremove}
and Theorems~\ref{thm-main-intro-unweighted} and~\ref{thm-main-Rn}.

However, we can now give the following example,
showing that sets of positive capacity (and positive measure) can be removable
for bounded quasiharmonic functions even when nonconstant quasiharmonic
functions 
are under consideration.

\begin{example} \label{ex-martio}
Assume that
$\R$ is unweighted.
Let $\Om=(-1,1)$ and $E=(-1,0]$.
Let $u$ be a bounded $Q$-quasiharmonic function on $\Om \setm E=(0,1)$.
Then $u$ is monotone and $u(0):=\lim_{x \to 0\limplus} u(x)$ exists.
By Martio's reflection principle~\cite[Theorem~4.1]{martioReflect}
the ``odd'' reflection
\begin{equation} \label{eq-refl}
     u^*(x)=\begin{cases}
         u(x), & 0 \le x < 1, \\
         2u(0)-u(-x), & -1 < x \le0,
       \end{cases}
\end{equation}
is a bounded $Q'$-quasiharmonic function on $\Om$.
Martio obtained the reflection principle with $Q'=2^p Q$.
This was improved to $Q'=\max\{2,2^{p-1}\}Q$
by Uppman~\cite[Theorem~2.3]{uppman}, who also showed
that the factor $\max\{2,2^{p-1}\}$ is best possible
independent of $Q$.
For 
\begin{equation} \label{eq-uppman}
     Q < \max \biggl\{\frac{1}{2-2^{1/p}},\frac{1}{2-2^{1/p-1}}\biggr\}
\end{equation}
he obtained the better bound $Q'=Q(2-Q^{-1/p})^p$, see
\cite[Theorem~2.4]{uppman}.
\end{example}

Thus we have obtained a removability result 
with an increase in the quasiharmonicity constant.
As there could be other extensions than those given by reflections, it
is not clear if the increase is really needed.

By reflecting several times one can extend $u$ to a quasiharmonic
function in any bounded open interval. However 
for unbounded intervals the increase (repeated infinitely many times) 
destroys the bounds for the quasiharmonicity constant  of the extended function
$\ut$,
unless $u$ is constant.
Moreover, $\ut$ would be unbounded,
by Proposition~6.6 in
Bj\"orn--Bj\"orn--Shanmugalingam~\cite{BBSliouville}.

Nevertheless, with repeated use of the reflection principle 
we can show the following characterization of removable sets
for bounded quasiharmonic functions on unweighted $\R$, under the same condition
as in Theorem~\ref{thm-R}.

\begin{thm} \label{thm-R-qharm}
The set $E$ is removable for bounded quasiharmonic  functions
  in $\Om \setm E$, with respect to unweighted $\R$,
  if and only if
  there is a constant $C$ such that
  for every component $I$ of $\Om$ it is
  true that  
  \begin{equation} \label{eq-I-q}
    I \setm E \text{ is connected}
    \quad \text{and} \quad
  |I| \le C |I \setm E|.
\end{equation}

More precisely, if $u$ is a bounded $Q$-quasiharmonic function on $\Om \setm E$
and \eqref{eq-I-q} holds with $C=2^N$, then $u$ has a
bounded $Q'$-quasiharmonic extension
to $\Om$, where
$      Q'=\max\{2,2^{p-1}\}^NQ$.

And conversely, if \eqref{eq-I-q} fails, then there is bounded \p-harmonic
function on $\Om \setm E$ with no bounded quasiharmonic extension to $\Om$.
\end{thm}  

When $Q$ satisfies \eqref{eq-uppman}
one can obtain a better (but more complicated) bound $Q'(Q,p,N)$ using
Theorem~2.4 in Uppman~\cite{uppman}, see
Example~\ref{ex-martio}.
In particular, $Q'(Q,p,N) \to 1$, as $Q \to 1$.
Moreover, $Q'(1,p,N)=1$ and we thus recover (the unweighted case of)
Theorem~\ref{thm-R}.

\begin{lem} \label{lem-fQ}
Let $Q \ge 1$ and $u(t)=t$, $0 <t< 1$.
Then there is an increasing function $f_Q:(1,\infty) \to \R$
such that $\lim_{x \to \infty} f_Q(x)=\infty$ and
such that if $U$ is a $Q$-quasiharmonic extension
of $u$ to $(0,x)$, where $x>1$,
then $U(x):=\lim_{t \to x\limminus} U(t) \ge f_Q(x)$.
\end{lem}

\begin{proof}
Let $U$ be a $Q$-quasiharmonic extension of $u$ to $(0,x)$, where $x>1$.
Then $U$ must be increasing and so $a=U(x):=\lim_{t \to x\limminus} U(t)$ exists.
Let $f_Q(x)= \inf a$ over all such extensions $U$,
which also must be an increasing function.
Next, let
\[
v(t)=\frac{at}{x}, \ 0 \le t \le x,
\quad \text{and} \quad
V(t)=\begin{cases}
t, &  0 \le  t \le 1, \\
\displaystyle 1+\frac{(a-1)(t-1)}{x-1}, &  1 \le  t \le x.
  \end{cases}
\]
Then $v$ is \p-harmonic in $(0,x)$ and $V$ is \p-harmonic in $(1,x)$.
Since $V(0)=U(0):=\lim_{x \to 0\limplus} U(t)$, $V(1)=U(1)$, $V(x)=U(x)$,
and $V$ minimizes the \p-energy among functions with the same
boundary values 
on $(0,1)$ as well as on  $(1,x)$,
we see that
\[
   \int_0^x (U')^p\,dt \ge \int_0^x (V')^p\,dt = 1+ \frac{(a-1)^p}{(x-1)^{p-1}}.
\]
On the other hand, by testing the quasiminimizer condition
on $(0,x)$ 
we also get that
\[
   \int_0^x (U')^p\,dt \le Q \int_0^x (v')^p\,dt = Q\frac{a^p}{x^{p-1}}.
\]
Hence
\[
a^p Q \ge x^{p-1}+\frac{(a-1)^p}{(x-1)^{p-1}}x^{p-1}
=  x^{p-1}+(a-1)^p \biggl(1-\frac{1}{x}\biggr)^{1-p}
\]
and thus also 
\begin{equation} \label{eq-fQ}
f_Q(x)^p Q \ge x^{p-1}+(f_Q(x)-1)^p \biggl(1-\frac{1}{x}\biggr)^{1-p}.
\end{equation}
Since the right-hand side in \eqref{eq-fQ} 
tends to $\infty$,
as $x \to \infty$, so does the left-hand side,
i.e.\ $\lim_{x \to \infty} f_Q(x)=\infty$.
\end{proof}

\begin{proof}[Proof of Theorem~\ref{thm-R-qharm}]
Assume first that the condition in the statement is satisfied with
constant $C=2^N$.
Let $u$ be a bounded $Q$-quasiharmonic function in $\Om \setm E$.
We may assume that $0 \le u \le 1$.
In each component $I$ of $\Om$, we can extend $u$ from $I \setm E$
to $I$ by using Martio's reflection principle (see Example~\ref{ex-martio})
at most $N$ times, resulting
in a function $U$ which is
$Q'$-quasiharmonic function on each $I$, and hence on $\Om$,
where $Q'=\max\{2,2^{p-1}\}^NQ$.
It follows from \eqref{eq-refl} that $\sup_I U - \inf_I U \le 2^N$
and hence that $|U| \le 2^N+1$ in $\Om$ (as $I \setm E \ne \emptyset$).

We now turn to the converse.
Assume first that 
there is a component $I$ of $\Om$ such that
$I \setm E$ is not connected.
Let $A$ be a component
of $I \setm E$ and $u=\chi_A$, which is a \p-harmonic function in $\Om \setm E$.
Now it is well known
that a quasiharmonic function in an open interval $I'$ which is constant
in an open subinterval $I'' \subset I'$ must be constant throughout $I'$,
see e.g.\ Martio--Sbordone~\cite[Lemma~2]{MaSb}.
Thus $u$ has no quasiharmonic extension to $\Om$.

Next, we consider the case when $I \setm E$ is connected
for every component $I$, but there is no constant $C$ as
in \eqref{eq-I-q}.
As in the proof of Theorem~\ref{thm-R}, we need to consider two cases.

\medskip

\emph{Case}~1.
\emph{There is a component $I$ of $\Om$ such that
$|I \setm E| < \infty = |I|$.}

In this case,
there is a nonconstant \p-harmonic function $u$ in $\Om \setm E$
which is identically $0$ outside $I$.
Let $U$ be any quasiharmonic extension of $u$ to $\Om$.
Then, by Lemma~\ref{lem-fQ}, $U$ must unbounded
in $I$. 

\medskip

\emph{Case}~2.
\emph{There is a sequence of components $I_j$ of $\Om$ such that
$|I_j| > j |I_j \setm E|$.}

In this case there is a \p-harmonic function $u$ in $\Om \setm E$
such that 
\[
   \inf_{I_j \setm E} u =0 \text{ and } 
   \sup_{I_j \setm E} u =1
   \quad \text{for } j=1,2,\ldots,
\]
and $u\equiv 0$ on $(\Om \setm \bigcup_{j=1}^\infty I_j) \setm E$. 
Let
$U$ be any $Q$-quasiharmonic extension of $u$ to $\Om$.
Then by Lemma~\ref{lem-fQ},
\[
   \sup_{I_j} U- \inf_{I_j} U
   \ge f_Q\biggl(\frac{|I_j|}{|I_j \setm E|}\biggr)
   \ge f_Q(j).
\]
Hence $U$ must be unbounded,
and since $Q$ was arbitrary, $u$ does not have any bounded
quasiharmonic extension to $\Om$.
\end{proof}

\section{Removability for \texorpdfstring{\A}{A}-superharmonic functions;
  \texorpdfstring{\\}
Theorem~\ref{thm-superh}}
\label{sect-super}

\emph{In addition to the general assumptions from the beginning
  of Section~\ref{sect-prelim}, throughout this section we assume that $n \ge 2$.}

\medskip

In this section we want to prove Theorem~\ref{thm-superh}.
For \A-superharmonic functions we could have used
the fact that they are \p-finely continuous
(see Section~12 and especially Theorem~12.8 in \cite{HeKiMa})
and at the same
time $\essliminf$-regularized.
The latter is part of the Definition~\ref{def-qsh}
also for quasisuperharmonic functions, but it is not known
if quasisuperharmonic functions are \p-finely continuous.
Instead, the following weaker type
of continuity, related to measure density, will be sufficient
for our purpose.

\begin{lem} \label{lem-density}
  Assume that $u$ is quasisuperharmonic in $\Om$, that $x \in \Om$
  and 
  that
  \[
  L=\limsup_{r \to 0} \frac{\mu(\{y \in B(x,r) : u(y) \ge \alp\})}
                    {\mu(B(x,r))} >0,
  \]
  then $u(x) \ge \alp$.
\end{lem}

\begin{proof}
We may assume that $\alp=0$.
Assume on the contrary that $u(x) <0$.
Let $\eps >0$.
As $u$ is lower semicontinuous there is $r_0>0$
such that $B(x,r_0) \subset \Om$ and $u(y)>u(x)-\eps $ for $y \in B(x,r_0)$.
Let $v=u-u(x)+\eps >0$ in $B(x,r_0)$.
Then there is $0 < r < r_0/5$ such that
\[
    \frac{\mu(\{y \in B(x,r) : u(y) \ge 0\})}
     {\mu(B(x,r))} > \frac{L}{2}.
\]
By the weak Harnack inequality for quasisuperharmonic functions in
Kinnunen--Martio~\cite[second display on p.~479]{KiMa03},
there are $\sigma>0$ and $c>0$ such
that 
\[
   |u(x)| \biggl(\frac{ L}{2}\biggr)^{1/\sigma}
   \le \biggl(\frac{1}{\mu(B(x,r))}\int_{B(x,r)} v^\sigma d\mu \biggr)^{1/\sigma}
   \le c \inf_{B(x,3r)} v
   \le c v(x) 
   =  c \eps.  
\]
Letting $\eps \to 0$ gives a contradiction, and
thus the result is proved.
\end{proof}

\begin{proof}[Proof of Theorem~\ref{thm-superh}]
If $\Cpmu(E)=0$,
then it follows from
\cite[Theorem~7.35]{HeKiMa}
and  Bj\"orn~\cite[Theorem~6.3]{ABremove}
that $E$ is removable as in \ref{s-pharm}--\ref{s-qharm-2},
with preserved quasisuperharmonicity constant in
\ref{s-qharm} and~\ref{s-qharm-2}.

Assume conversely that $\Cpmu(E)>0$.
We shall construct a bounded \A-super\-harmonic function $u$ on $\Om \setm E$
which has no quasisuperharmonic extension to $\Om$ (neither bounded nor unbounded),
from which the nonremovability in all four senses \ref{s-pharm}--\ref{s-qharm-2}
follows.
We begin with some reductions.

First of all, we can find a component $\Omt$ of $\Om$ such that $\Cpmu(E \cap \Omt)>0$.
If we then can construct a bounded \A-superharmonic function $\ut$ in $\Omt \setm E$
which has no quasisuperharmonic extension to $\Omt$, then it follows
that
\[
  u=\begin{cases}
    \ut, & \text{in } \Omt \setm E, \\
    0, & \text{in } (\Om \setm \Omt) \setm E,
    \end{cases}
\]
is a bounded \A-superharmonic function in $\Om \setm E$
without quasisuperharmonic extension to $\Om$.
We may thus assume, without loss of generality, that $\Om$ is connected.

Let
\[
   E_0=\{x \in E : \Cpmu(E \cap B)=0 \text{ for some ball } B \ni x\}.
\]
Then $E_0$ is a relatively open subset of $E$ which
can be written as a countable union of sets of capacity zero,
and thus $\Cpmu(E_0)=0$.
By the first part of the proof, $E_0$ is thus removable in all four senses,
and since quasisuperharmonic functions are $\essliminf$-regularized,
the extensions are unique and are also bounded resp.\ bounded from below.
Thus we may first remove $E_0$ and proceed by showing that $E \setm E_0$
is not removable.
We may thus assume, without loss of generality, that $E_0 = \emptyset$.

Next, let $z \in E\cap \bdy E$.
Then $\dist(z,\bdy \Om)>0$ and we can find $\xi \in \Om$ such
that $\dist(\xi,z) < \tfrac14 \min\{\dist(z,\bdy \Om),1\}$.
(If $\Om = \R^n$, we consider $\dist(z,\bdy \Om)$ to be $\infty$.)
Finally, we find $\z \in E$ so that $|\z-\xi|=\dist(\xi,E)<\dist(\xi,\bdy \Om)$.
We may assume that $\z=0$ and $\xi=(1,0,\ldots,0)$.

It follows that the open interval $(0,\xi) \subset \Om \setm E$
and that $B(0,1) \subset \Om$.
We now consider two cases.
As in  Section~\ref{sect-Rn},
we will use \A-harmonic Perron solutions
and boundary
regularity.

\medskip

\emph{Case}~1. \emph{There is $0<r<1$ such that $\{x \in E : |x|=r\}=\emptyset$.}

Let $K=E \cap B(0,r)$, which is a compact set with positive capacity
(as $E_0=\emptyset$), and $G=B(0,r) \setm K$.
Then $\chi_{\bdy B(0,r)} \in C(\bdy G)$ and we can define
\[
   u=\begin{cases}
    \Hpind{G} \chi_{\bdy B(0,r)}, & \text{in } G, \\
    1, & \text{in } \Om \setm (E \cup G),
   \end{cases}
\]
where $\Hpind{G} \chi_{\bdy B(0,r)}$ is an \A-harmonic Perron solution.
Then $u$ is \A-harmonic in $G$.
Moreover, $u$ is continuous in $\Om \setm E$ because all
the points in $\bdy B(0,r)$ are regular for $G$.
It thus follows from the pasting lemma~\cite[Lemma~7.9]{HeKiMa}
that $u$ is \A-superharmonic in $\Om \setm E$.

As $\Cpmu(K)>0$, the Kellogg property shows that $u$ is not constant.
If $u$ had a quasisuperharmonic extension $\ut$ to $\Om$
it would attain its minimum in $\itoverline{B(0,r)} \subset \Om$,
and that minimum must be $<1$ and lie in $B(0,r)$.
But this contradicts the strong minimum principle for
quasisuperharmonic functions
(which 
for quasisuperharmonic functions follows 
from the weak Harnack inequality in
Kinnunen--Martio~\cite[second display on p.~479]{KiMa03}).
This case is thus settled.

\medskip

\emph{Case}~2. \emph{$\{x \in E : |x|=r\}\ne \emptyset$ for $0<r<1$.}

For each $x \in B(0,r) \setm \{0\}$, let $\theta_x=\arccos (\xi \cdot x/|x|)$
be the angle which $x$ makes with $\xi$.
For each $j=1,2,\ldots$\,, there is
\[
  x_j \in K_j:=\{x \in E : 4^{-j} \le |x| \le 2 \cdot 4^{-j}\}
  \quad \text{such that} \quad
  \theta_{x_j}=\min\{\theta_x : x \in K_j\}.
\]
Since $0$ is a closest point to $\xi$ in $E$, we must have
$\theta_{x_j} > \pi/4$.
Let $B_j=B(x_j,4^{-j-2})$.
Then there is $y_j \in B_j$ such that
$B_j':=B(y_j,4^{-j-4}) \subset B_j \cap (\Om \setm E)$.

Let $A=\bigcup_{j=1}^\infty \itoverline{B}_j$, 
$G=B(0,1) \setm (E \cup A)$,
\[
   f(x)=(1-4^2(|x|-1))_\limplus +\sum_{j=1}^\infty (1-4^{j+2}\dist(x,B_j))_\limplus
\]
(which is continuous except at $x=0$)
and
\[
   u=\begin{cases}
    \uHpind{G} f, & \text{in } G, \\
    1, & \text{in } \Om \setm (E \cup G),
   \end{cases}
\]
where $\uHpind{G} f$ is an \A-harmonic upper Perron solution.
All the points in $\bdy G \cap (\bdy B(0,1) \cup \bigcup_{j=1}^\infty \bdy B_j)$
are regular with respect to $G$,
and hence, by \cite[Lemma~9.6]{HeKiMa}, $u$ is continuous in $\Om \setm E$.
It thus follows from the pasting lemma~\cite[Lemma~7.9]{HeKiMa}
that $u$ is \A-superharmonic in $\Om \setm E$.

Assume that
$u$ has a quasisuperharmonic extension $U$ to $\Om$.
Since $U=u\equiv 1$ in $\bigcup_{j=1}^\infty B'_j$,
it follows from Lemma~\ref{lem-density} and
Bj\"orn--Bj\"orn~\cite[Lemma~3.3]{BBbook} that
$U(0) \ge 1$.
On the other hand,
because $\{x \in E : |x|=3\cdot 4^{-j}\}\ne \emptyset$ and $E_0 = \emptyset$,
it follows from the Kellogg property that there is a
regular point $y_j \in \bdy G$ with
$\frac 52 \cdot 4^{-j} < |y_j| <   \frac 72 \cdot 4^{-j}$, $j=1,2,\ldots$\,.
Thus $f_j(y_j)=0$ (again using \cite[Lemma~9.6]{HeKiMa})
and $\lim_{G \ni x \to y_j} u(x)=0$.
Since $\lim_{j \to \infty} y_j = 0$ and $U$ is lower semicontinuous,
it follows that $U(0) \le 0$, contradicting the above.
Thus also this case is settled.
\end{proof}

\end{document}